\documentclass[11pt]{amsart}
\usepackage[applemac]{inputenc}
\usepackage[T1]{fontenc}
\usepackage{lmodern}
\usepackage{amsmath,amssymb}
\usepackage[all]{xy}
\usepackage{a4wide}

\usepackage{amsthm}
\usepackage{amssymb}
\usepackage{amsmath}

\theoremstyle{plain}
\newtheorem{thm}{Theorem}[section]
\newtheorem{prop}[thm]{Proposition}
\newtheorem{lemma}[thm]{Lemma}

\theoremstyle{definition}
\newtheorem{defi}[thm]{Definition}

\theoremstyle{remark}
\newtheorem*{rem}{Remark}

\usepackage{cite}

\title{Moduli of Autodual Instanton Bundles }
\author{Marcos Jardim, Simone Marchesi, Anna Wi\ss dorf}

\DeclareMathOperator{\im}{im}
\DeclareMathOperator{\coker}{coker}
\DeclareMathOperator{\End}{End}
\DeclareMathOperator{\Hom}{Hom}
\DeclareMathOperator{\rk}{rk}
\newcommand{\KK}{{\mathbb K}}
\newcommand{\PP}{{\mathbb P}}
\newcommand{\OO}{{\mathcal O}}
\newcommand{\EE}{{\mathcal E}}
\newcommand{\FF}{{\mathcal F}}

\begin{document}

\begin{abstract}
We provide a description of the moduli space of framed autodual instanton bundles on projective space, focusing on the particular cases of symplectic and orthogonal instantons. Our description will use the generalized \emph{ADHM} equations which define framed instanton sheaves.
\end{abstract}

\maketitle

\section*{Introduction}
Instanton bundles, which represent an important link between algebraic geometry and mathematical physics, were first introduced in 1978 by Atiyah, Drinfeld, Hitchin and Manin on the projective space $\PP^3$ (see \cite{ADHM}), and they were further generalized in 1986 for projective spaces of odd dimension by Okonek and Spindler in \cite{OS}, motivated by \cite{MCS}.

Since then, the study of these objects and their moduli spaces has been of continued interest for various authors. In dimension $3$, the nonsingularity and irreducibility of the moduli space of rank $2$ instantons bundles on $\PP^3$ remained open for many years, see for instance \cite[Introduction]{CTT} and the references therein. More recently, Tikhomirov proved in \cite{Tik1,Tik2} that the moduli space of rank $2$ instanton bundles on $\PP^3$ with arbitrary second Chern class is irreducible; Markushevich and Tikhomirov proved that it is rational; Jardim and Verbitsky proved in \cite{JV} that such moduli spaces are always nonsingular.

For higher dimensional projective spaces or higher rank, the moduli space of instanton bundles become much more complicated, so many authors have considered instanton bundles with some additional structure. To name just a few, Spindler and Trautmann studied the moduli of special instanton bundles in \cite{ST}, while Costa and Ottaviani proved in \cite{CO2} that the moduli space of symplectic instanton bundles on $\PP^{2n+1}$ is affine. More recently, Bruzzo, Markushevich and Tikhomirov identified an irreducible component of expected dimension in the moduli space of symplectic instanton bundles on $\PP^3$ \cite{BMT}; Costa, Hoffmann, Mir\'o-Roig and Schmitt further study symplectic instanton bundles on $\PP^{2n+1}$ \cite{CHMRS}. On the other hand, Farnik, Frapporti and Marchesi proved in \cite{orth} that there are no orthogonal instantons bundles of rank $2n$ on $\PP^{2n+1}$.

In this paper we consider \emph{autodual instantons} of arbitrary rank on projective spaces, focusing in particularly on symplectic and orthogonal instantons. To be more precise, an autodual instanton is a pair $(\EE,\varphi)$ consisting of an instanton bundle $\EE$ and an isomorphism $\varphi:\EE\to\EE^\vee$. Note that $\EE$ is symplectic when $\varphi=-\varphi^\vee$ and orthogonal when $\varphi=\varphi^\vee$.

We first describe the moduli space of framed autodual instanton bundles using the \emph{ADHM} construction introduced by Henni, Jardim and Martins in \cite{HJM}. This allows us to provide a description of the moduli spaces of symplectic and orthogonal instantons using matrices, a description that is useful in establishing existence and non-existence results for orthogonal instantons.

The paper is organized as follows. In Section \ref{prel} we will fix some notations and definitions. Section \ref{auto} is dedicated to the investigation of the matrices involved in the \emph{ADHM} construction of framed autodual instanton bundles and the description of their moduli spaces in terms of such matrices. In Sections \ref{sym} and \ref{ort} we will deal  with the symplectic and orthogonal cases respectively. 

It turns out that orthogonal instanton bundles are somewhat hard to find, as indicated by the Farnik--Frapporti--Marchesi non-existence result mentioned above. We prove, in addition, that there are no orthogonal instanton bundles with odd second Chern class and no orthogonal instanton bundles of rank $2$ and charge $2$ on $\PP^2$. However, we provide examples of orthogonal instantons of rank $2$ and charges $4$ and $6$ on $\PP^2$, and indicate how to obtain examples of charge $2k$ for every $k\ge4$. Finally, we also provide explicit examples of orthogonal instantons of rank $4$ and charge $2$ on $\PP^3$.

\paragraph{\bf Acknowledgments.}
MJ is partially supported by the CNPq grant number 302477/2010-1 and the FAPESP grant number 2011/01071-3. 
SM is supported by the FAPESP post-doctoral grant number 2012/07481-1. We also thank the University of Campinas and Freie Universit\"at Berlin for financial support.


\section{Framed instanton bundles}\label{prel}

In this section we fix the notation used throughout this paper and recall some definitions. We work over an algebraically closed field of characteristic $0$, which we denote by $\KK$. We will denote by $\EE^\lor$ the dual of a sheaf (or bundle) $\EE$, by $V^\lor$ the dual of the vector space $V$ over $\KK$ and by $\alpha^\lor : \FF^\lor \longrightarrow \EE^\lor$ the dual of the sheaf (or bundle) map $\alpha: \EE \longrightarrow \FF$. We use the same notation for duality of linear applications of vector spaces.

\begin{defi}[Instanton Sheaf]
A coherent sheaf $\mathcal{E}$ on $\mathbb{P}^n$ is called an \textit{instanton sheaf} of charge $c$ and rank $r$ if it is defined as the cohomology of a linear monad
 \[
\begin{xy}
 \xymatrix{
  \OO^{\oplus c}_{\PP^n}(-1) \ar[r]^a & \OO^{\oplus r+2c}_{\PP^n} \ar[r]^b & \OO^{\oplus c}_{\PP^n}(1),
  }
\end{xy}
\]
where, by definition of a monad, we have $b\circ a=0$ and $a$ (respectively $b$) is an injective (respectively surjective) sheaf map. 
Moreover, if $\EE$ is locally free, we call it an \textit{instanton bundle}.
\end{defi}

The total Chern class of an instanton sheaf is then $c(\EE)=\frac{1}{(1-t^2)^c}$. In particular, we have $c_1(\EE)=0$ and $c_2(\EE)=c>0$.

  Alternatively, there is the following cohomological characterization \cite{J-i}. A torsion-free sheaf $\EE$ on $\PP^n$ ($n\geq 2$) is an instanton sheaf if $c_1(\EE)=0$ and
\begin{enumerate}
 \item $H^0(\EE(-1))=H^n(\EE(-n))=0$;
 \item $H^1(\EE(-2))=H^{n-1}(\EE(1-n))=0$, if $n\ge 3$;
 \item $H^p(\EE(k))=0$ for $2\leq p\leq n-2$ and all $k$, if $n\geq 4$.
\end{enumerate}

\begin{defi}[Framing]
We say that a coherent sheaf $\EE$ on $\PP^n$ is of \emph{trivial splitting type} if there is a line $l\subset \PP^n$ such that $\EE_{|l} \simeq \OO_{l}^{\oplus r}$; a \emph{framing at $l$} is the choice of an isomorphism $\Phi : \EE_{|l} \longrightarrow \OO_{l}^{\oplus r}$. Fixing the line $l$, a \emph{framed instanton sheaf} is a pair $(\EE,\Phi)$ consisting of an instanton sheaf $\EE$ of trivial splitting type and a framing $\Phi$ at the line $l$.

A morphism of framed instanton sheaves $f\colon (\EE_1,\Phi_1)\rightarrow (\EE_2,\Phi_2)$ is a morphism of sheaves $f\colon\EE_1\rightarrow\EE_2$ that respects the framing, in other words
\[
\begin{xy}
 \xymatrix{
 (\EE_1)_{|l} \ar[r]^{\Phi_1}\ar[d]_{f_{|l}} & \OO_{l}^{\oplus r_1} \ar[d]^{\bar{f}} \\
 (\EE_{2})_{|l} \ar[r]^{\Phi_2} & \OO_{l}^{\oplus r_2} 
 }
\end{xy}
\]
is a commutative diagram.
\end{defi}

\subsection*{\emph{ADHM} construction of instanton bundles}

Let us recall the \emph{ADHM} construction for framed instanton bundles, whose details can be found in \cite{HJM}.

Let $V$ and $W$ be $\KK$-vector spaces of dimension $\dim V=c$ and $\dim W = r$. An $(n-2)$-\emph{dimensional ADHM datum} is given by linear maps $(A_k, B_k, I_k, J_k)$, for $k=0,\ldots,n-2$, where $A_k, B_k\in \End(V)$, $I_k\in \Hom(W,V)$ and $J_k\in \Hom(V,W)$.\\
Choose homogeneous coordinates $[x:y:z_0:\ldots:z_{n-2}]$ on $\PP^n$ and define
\begin{align*}
 A := \sum_{i=0}^{n-2}A_iz_i, \:\:\:
 B:= \sum_{i=0}^{n-2}B_iz_i ,  \:\:\:I := \sum_{i=0}^{n-2}I_iz_i, \:\:\:J := \sum_{i=0}^{n-2}J_iz_i.
\end{align*}

Set also $\mathbb{B} := (\End(V)\oplus\End(V)\oplus\Hom(W,V)\oplus\Hom(V,W))\otimes H^0(\PP^{n-2},\OO_{\PP^{n-2}}(1))$,
and consider the map
\begin{align*}
 \mu\colon\mathbb{B}\rightarrow \End(V)\otimes H^0(\PP^{n-2},\OO_{\PP^{n-2}}(2)) \\
 (A,B,I,J)\mapsto [A,B]+IJ.
 \end{align*}
Here, we let $z_0,\ldots,z_{n-2}$ be a basis of $H^0(\PP^{n-2},\OO_{\PP^{n-2}}(1))$. 

\begin{defi}($(n-2)$-dimensional \emph{ADHM} datum)
  The subset $\mu^{-1}(0)\subset\mathbb{B}$ is called the set of all $(n-2)$-\emph{dimensional ADHM data}. An element $X=(A,B,I,J)\in\mu^{-1}(0)$ satisfies the \emph{ADHM} equation $[A,B] + IJ = 0$.
 \end{defi}
 
Given an element $X=(A,B,I,J)\in\mathbb{B}$ we define $\widetilde{W}:= V\oplus V\oplus W$ and a sequence
\begin{equation}\label{eq-mon}
\begin{xy}
 \xymatrix{
  V\otimes \OO_{\PP^n}(-1) \ar[r]^{\alpha} & \widetilde{W}\otimes\OO_{\PP^n} \ar[r]^{\beta} & V\otimes \OO_{\PP^n}(1)
  }
\end{xy}
\end{equation}
where the maps involved are given by
\[
 \alpha=\left( \begin{matrix}
                A+\mathbf{1}x \\
                B+\mathbf{1}y \\
                J
               \end{matrix}
\right), \
\beta=\left( \begin{matrix}
              -B-\mathbf{1}y & A+\mathbf{1}x & I
             \end{matrix}
 \right).
\]
It is straightforward to check that the vanishing of the composition $\beta\circ\alpha=0$ is equivalent to the property $X=(A,B,I,J)\in\mu^{-1}(0)$, i.e. $X$ being an \emph{ADHM} datum. We are two steps away for (\ref{eq-mon}) to define an instanton bundle, therefore we need a little bit more.

\begin{defi}[Regularity]
 An \emph{ADHM} datum $X=(A,B,I,J)$ is called \textit{globally regular}, if the fibre map $\alpha_p$ is injective for every $p\in\PP^{n}$ and the fibre map $\beta_p$ is surjective for every $p\in\PP^{n}$.
\end{defi}

\begin{rem}
The Definition above is different from \cite[Definition 2.2]{HJM}; however, both definitions are equivalent by \cite[Proposition 3.3]{HJM}.  
\end{rem}

Regularity of $X$ guarantees that the cohomology sheaf $\EE=\ker(\beta)/\im(\alpha)$ is a locally free sheaf. 
 
Instanton bundles defined by monads arising from \emph{ADHM} data will always have trivial splitting type. 
In fact, they are trivial on the line $l=\left\{ z_0=\ldots =z_{n-2}=0 \right\}$ with fibre given by $W$, i.e. $\EE_{|l}\cong W\otimes\OO_l$; this isomorphism provides a framing at the line $l$. 

We let $GL(V)$ act on $\mathbb{B}$ by defining
\[ g.(A,B,I,J) = \left(\sum_{i=0}^{n-2} gA_ig^{-1}, \sum_{i=0}^{n-2} gB_ig^{-1}, \sum_{i=0}^{n-2} gI_i, \sum_{i=0}^{n-2} J_ig^{-1} \right) . \]
One checks that the set $\mu^{-1}(0)$ is $GL(V)$-invariant. This enables us to identify \emph{ADHM} data that produce isomorphic cohomology bundles.

\begin{thm}[Theorems 3.8 and 4.2 in \cite{HJM}]
There is a 1-1-correspondence between the following objects:
 \begin{itemize}
  \item $GL(V)$-orbits of globally regular solutions $(A,B,I,J)$ to $[A,B]+IJ=0$;
  \item isomorphism classes of framed instanton bundles with rank $r=\dim W$ and charge $c=\dim V$.
 \end{itemize}
\end{thm}


\section{Autodual instanton bundles}\label{auto}

We shall now state the precise definitions of the objects in which we are interested.

\begin{defi}[Autoduality]
We call an instanton bundle $\EE$ \emph{autodual} if it is isomorphic to its dual, i.e. there exists an isomorphism  $\varphi\colon\EE\rightarrow\EE^\lor$. If, in addition, $\EE$ is framed by $\Phi$, and the isomorphism $\varphi$ preserves the framing, then we say that $(\EE,\Phi)$ is a framed autodual instanton bundle.
\end{defi}

We  denote a framed autodual instanton bundle by the triple $(\EE,\Phi,\varphi)$, where $\Phi\colon \EE_{|l}\rightarrow W\otimes\OO_{l}$ denotes the framing at the line $l$ and $\varphi\colon\EE\rightarrow\EE^\lor$ the isomorphism that respects the framing.

Let $\mathcal{F}_{\PP^n}^{a}(r,c)$ denote the set of isomorphism classes of framed autodual instanton bundles $(\EE,\Phi,\varphi)$ of rank $r$ and second Chern class $c_2(E)=c$. In the sequel, we will also consider unframed autodual instantons bundles (which are of trivial splitting type, but no framing is fixed), we denote their set of isomorphism classes by $\mathcal{M}_{\PP^n}^{a}(r,c)$. Obviously, there is the map $\mathcal{F}_{\PP^n}^{a}(r,c)\rightarrow\mathcal{M}_{\PP^n}^{a}(r,c)$ that forgets the framing.

The purpose of this section is to see how an isomorphism $\varphi\colon\EE\rightarrow\EE^\lor$ of an instanton bundle to its dual bundle is reflected in the corresponding monad and the \emph{ADHM} data. We will conclude the section by describing the moduli space of framed autodual instanton bundles using the relations obtained.

In \cite[Proposition 2.21]{AO} it is proven that the dual of an instanton bundle is again an instanton bundle. However, we will need the following stronger result. 

\begin{lemma}
If $\EE$ is an instanton bundle, then its dual bundle $\EE^\lor$ is the cohomology of the monad  which is dual to the monad that defines $\EE$. In particular, $\EE^\lor$ is again instanton.
\end{lemma}

\begin{proof}
Considering the monad which defines $\EE$
 \begin{equation}\label{
 }
\begin{xy}
 \xymatrix{
  \OO_{\PP^n}^{\oplus c}(-1) \ar[r]^{\alpha} & \OO_{\PP^n}^{\oplus r+2c} \ar[r]^{\beta} & \OO_{\PP^n}^{\oplus c}(1),
  }
\end{xy}
\end{equation}
we can split it in two pairs of short exact sequences of vector bundles, respectively
\begin{equation}\label{eq-I}
  \left\{  
  \begin{array}{ll}
  0 \longrightarrow K \longrightarrow \OO_{\PP^n}^{\oplus r+2c} \stackrel{\beta}{\longrightarrow} \OO_{\PP^n}^{\oplus c}(1) \longrightarrow 0,
& K= \ker\beta \vspace{3mm}\\
  0 \longrightarrow  \OO_{\PP^n}^{\oplus c}(-1) \stackrel{\alpha}{\longrightarrow}  K \longrightarrow  \EE \longrightarrow 0,
& \EE=\ker\beta/\im\alpha
\end{array}
 \right.
\end{equation}
and
\begin{equation}\label{eq-II}
  \left\{  
  \begin{array}{ll}
  0 \longrightarrow \OO_{\PP^n}^{\oplus c}(-1) \stackrel{\alpha}{\longrightarrow}  \OO_{\PP^n}^{\oplus r+2c} \longrightarrow Q  \longrightarrow 0,
& Q= \coker\alpha \vspace{3mm}\\
  0 \longrightarrow \EE \longrightarrow Q \longrightarrow \OO_{\PP^n}^{\oplus c}(1) \longrightarrow 0.
          \end{array} 
 \right.
\end{equation}
If $\EE$ is locally free, we can dualise (\ref{eq-I}) and we obtain
\begin{equation}\label{eq-III}
 \left\{  \begin{array}{ll}
  0 \longrightarrow \OO_{\PP^n}^{\oplus c}(-1) \stackrel{\beta^\lor}{\longrightarrow}  \OO_{\PP^n}^{\oplus r+2c} \longrightarrow K^\lor \longrightarrow 0\vspace{3mm}\\
  0 \longrightarrow \EE^\lor \longrightarrow K^\lor \stackrel{\alpha^\lor}{\longrightarrow}  \OO_{\PP^n}^{\oplus c}(1) \longrightarrow 0.
          \end{array} 
 \right.
\end{equation}
Comparing (\ref{eq-III}) with (\ref{eq-II}) and taking in consideration the isomorphism $(\ker\beta)^\lor\cong\coker\beta^\lor$, it is possible to reconstruct the following monad
\begin{equation}\label{eq-dualmon}
\begin{xy}
 \xymatrix{
  \OO_{\PP^n}^{\oplus c} (-1) \ar[r]^{\beta^\lor} & \OO_{\PP^n}^{\oplus r+2c} \ar[r]^{\alpha^\lor} & \OO_{\PP^n}^{\oplus c} (1).
  }
\end{xy}
\end{equation}
We therefore conclude that when $\EE$ is the bundle defined by (\ref{eq-mon}), its dual $\EE^\lor$ is the cohomology of the monad (\ref{eq-dualmon}), which is dual to (\ref{eq-mon}), and this proves the statement.
\end{proof}

\begin{rem}
The same argument works for more general monads of locally free sheaves
  \[
\begin{xy}
 \xymatrix{
  \mathcal{A} \ar[r]^{\alpha} & \mathcal{B} \ar[r]^{\beta} & \mathcal{C}
  }
\end{xy}
\]
whose cohomology is locally free. However, in general, the argument is false if $\mathcal{E}$ is not locally free, because dualising (\ref{eq-I}) would not give (\ref{eq-III}) in this case.
\end{rem}

Let us return to our case.  Recall that if we consider the vector bundles $\EE$ and $\EE'$ defined respectively as the cohomology of the monads 
$$
\xymatrix{
M: &  \OO_{\PP^n}^{\oplus c} (-1) \ar[r] & \OO_{\PP^n}^{\oplus r+2c} \ar[r] & \OO_{\PP^n}^{\oplus c} (1)
  }
\
$$
and
$$
\xymatrix{
M': &  \OO_{\PP^n}^{\oplus c'} (-1) \ar[r] & \OO_{\PP^n}^{\oplus r'+2c'} \ar[r] & \OO_{\PP^n}^{\oplus c'} (1),
  }
\
$$
we obtain a bijection 
$$
\Hom(\EE,\EE') \longrightarrow \Hom(M,M')
$$
between homomorphisms of bundles and homomorphisms of monads (see for example Lemma II.4.1.3 in \cite{oss}). Hence, an isomorphism $\varphi\colon\EE\rightarrow\EE^\lor$ can be lifted to an isomorphism of monads 

\[
\begin{xy}
 \xymatrix{
 V\otimes \OO_{\PP}(-1) \ar[r]^{\alpha}\ar[d]^{G_1} & \widetilde{W}\otimes\OO_{\PP} \ar[r]^{\beta}\ar[d]^{F} & V\otimes \OO_{\PP}(1) \ar[d]^{G_2} \\
 V^\lor\otimes \OO_{\PP}(-1) \ar[r]^{\beta^\lor} & \widetilde{W}^\lor\otimes\OO_{\PP} \ar[r]^{\alpha^\lor} & V^\lor\otimes \OO_{\PP}(1).
 }
\end{xy}
\]

Remember that $\widetilde{W}=V\oplus V\oplus W$, so we are allowed to write $F$ into block form

\[
F=\left(\begin{matrix}
                 F_1 & F_2 & F_3 \\
                 F_4 & F_5 & F_6 \\
                 F_7 & F_8 & F_9
\end{matrix}
\right).
\]

The commutativity of the diagram gives us relations for the maps involved, in particular the left half gives us $F\alpha=\beta^\lor G_1$ and the right half gives us $G_2\beta=\alpha^\lor F$. Using the description of $\alpha$ and $\beta$ given by the \emph{ADHM} construction and the description of $F$ through blocks, we are able to compute the following relations.

First of all observe that the bundle morphisms $G_1, G_2$ are given by vector spaces maps, which we will denote the same way, $G_1, G_2 \colon V\rightarrow V^\lor$. Commutativity of the left half reduces to

\begin{align*}
 F\alpha &= \left( \begin{matrix}
                 F_1 & F_2 & F_3 \\
                 F_4 & F_5 & F_6 \\
                 F_7 & F_8 & F_9
               \end{matrix}
 \right) 
 \left( \begin{matrix}
              A+\mathbf{1}x \\
              B+\mathbf{1}y \\
              J
             \end{matrix}
 \right) =
 \left( \begin{matrix}
              F_1A+F_1x+F_2B+F_2y+F_3J \\
              F_4A+F_4x+F_5B+F_5y+F_6J \\
              F_7A+F_7x+F_8B+F_8y+F_9J
             \end{matrix}
 \right),
  \\
 \beta^\lor G_1&=\left( \begin{matrix}
              -B^\lor-\mathbf{1}y \\
              A^\lor+\mathbf{1}x \\
              I^\lor
             \end{matrix}
 \right)G_1 =
 \left( \begin{matrix}
              -B^\lor G_1-G_1y \\
              A^\lor G_1+G_1x \\
              I^\lor G_1
             \end{matrix}
 \right),
\end{align*}

from which we obtain $F_1=F_5=F_7=F_8=0$ and $F_2=-G_1$, $F_4=G_1$. The right half then reduces to

\begin{align*}
 \alpha^\lor F &= \left( \begin{matrix}
              A^\lor+\mathbf{1}x & B^\lor+\mathbf{1}y & J^\lor
             \end{matrix}
 \right) 
 \left( \begin{matrix}
                 0 & -G_1 & F_3 \\
                 G_1 & 0 & F_6 \\
                 0 & 0 & F_9
               \end{matrix}
 \right) \\
 &=\left( \begin{matrix}
              B^\lor G_1+G_1y & -A^\lor G_1-G_1x & A^\lor F_3+F_3x+B^\lor F_6+F_6y+J^\lor F_9
             \end{matrix}
 \right),
  \\
 G_2\beta&= G_2 \left( \begin{matrix}
              -B-\mathbf{1}y & A+\mathbf{1}x & I
             \end{matrix}
 \right) =
 \left( \begin{matrix}
              -G_2B-G_2y & G_2A+G_2x & G_2I
             \end{matrix}
 \right),
\end{align*}
and we get $F_3=F_6=0$ and furthermore $G_1=-G_2$.

The following result is a straightforward consequence of the previous considerations.

\begin{lemma}\label{lem mor} 
A framed autodual instanton bundle is given by a morphism of linear monads
\begin{equation}\label{duality diagram}
\begin{xy}
 \xymatrix{
 V\otimes \OO_{\PP}(-1) \ar[r]^{\alpha}\ar[d]^{-G} & \widetilde{W}\otimes\OO_{\PP} \ar[r]^{\beta}\ar[d]^{F} & V\otimes \OO_{\PP}(1) \ar[d]^{G} \\
 V^\lor\otimes \OO_{\PP}(-1) \ar[r]^{\beta^\lor} & \widetilde{W}^\lor\otimes\OO_{\PP} \ar[r]^{\alpha^\lor} & V^\lor\otimes \OO_{\PP}(1)
 }
\end{xy}
\end{equation}
where $$F=\left( \begin{matrix}
                 0 & G & 0 \\
                 -G & 0 & 0 \\
                 0 & 0 & H
               \end{matrix}
 \right)\colon V\oplus V\oplus W\rightarrow (V\oplus V\oplus W)^\lor $$ with isomorphisms  $G\colon V\rightarrow V^\lor$ and $H\colon W\rightarrow W^\lor$. Furthermore, the data $(A,B,I,J,G,H)$ satisfies the following \emph{duality relations}:

\begin{equation}\label{duality}
\begin{aligned}
 GB &= B^\lor G  \\
 GA &= A^\lor G \\
 HJ &= -I^\lor G \\
 GI &= J^\lor H.
\end{aligned} 
\end{equation}
\end{lemma}
\begin{proof}
It only remains for us to observe that the duality relations follow from the commutativity of the diagram (\ref{duality diagram}).
\end{proof}

Next, in order to describe moduli spaces of framed autodual bundles, we need to extend the $GL(V)$-action on the \emph{ADHM} data $(A,B,I,J)$ to the extended data $(A,B,I,J,G,H)$. 

Indeed, a change of coordinates $g\in GL(V)$ induces the change of coordinates in the dual vector space given by $g^\lor\in GL(V^\lor)$. The action will be given by
\[ g.G := (g^\lor)^{-1}Gg^{-1}\colon V\rightarrow V^\lor, \]
such that the action of $GL(V)$ on the extended data is
\[g.(A,B,I,J,G,H) = (gAg^{-1},gBg^{-1},gI,Jg^{-1},(g^\lor)^{-1}Gg^{-1},H).  \]

In light of everything seen until now, we can state the following result.

\begin{prop}\label{prop-modaut}
 For fixed $r>0$ and $c>0$ we have a bijection
 \[ \mathcal{F}_{\PP^n}^{a}(r,c) \cong \left\{ (A,B,I,G,H) \right\}/_{GL(V)} \]
 where $(A,B,I,G,H)$ satisfies the following:
 \begin{itemize}
  \item $(A,B,I,-H^{-1}I^\lor G)$ is a regular ADHM datum.
  \item $GA=A^\lor G$, $GB=B^\lor G$, $GIH^{-1}+G^\lor I(H^\lor)^{-1}=0$.
 \end{itemize}

\end{prop}
\begin{proof}
 We have already observed that fixing a framed autodual instanton bundle gives us an isomorphism between the defining monads, therefore we get the extended \emph{ADHM} data $(A,B,I,J,G,H)$ satisfying the duality relations (\ref{duality}).
 
Observe also that $J=-H^{-1}I^\lor G$ by the duality relations and thus $GIH^{-1}-J^\lor=0$, which gives the last relation. One checks that all relations are $GL(V)$-invariant. Furthermore, the isomorphism $H\colon W\rightarrow W^\lor$ fits into the commutative diagram
\begin{equation}\label{H-Phi}
\begin{xy}
 \xymatrix{
 \EE_{|l} \ar[r]^{\Phi}\ar[d]_{\varphi|_l} & W\otimes\OO_{l} \ar[d]^{H} \\
 \EE^\lor_{|l} \ar[r]^{(\Phi^\lor)^{-1}} & W^\lor\otimes\OO_{l} 
 }
\end{xy}
\end{equation}
which determines the framing of $\EE^\lor$.
\end{proof}

Two particular classes of autodual instanton bundles have been previously considered in the literature, and will be our focus for the remainder of this paper: symplectic and orthogonal instanton bundles.

\begin{defi}
\begin{enumerate}
\item A \textit{framed symplectic instanton bundle} $(\mathcal{E},\Phi,\varphi)$ is a framed autodual instanton bundle $(\mathcal{E},\Phi,\varphi)$ such that the isomorphism $\varphi\colon\mathcal{E}\rightarrow\mathcal{E}^\lor$ satisfies $\varphi^\lor=-\varphi$. We denote by $\mathcal{F}_{\PP^n}^{s}(r,c)$ the moduli space of framed symplectic instanton bundles and by $\mathcal{M}_{\PP^n}^{s}(r,c)$ the moduli space of symplectic instanton bundles of trivial splitting type.
\item A \textit{framed orthogonal instanton bundle} $(\mathcal{E},\Phi,\varphi)$ is a framed autodual instanton bundle $(\mathcal{E},\Phi,\varphi)$ such that the isomorphism $\varphi\colon\mathcal{E}\rightarrow\mathcal{E}^\lor$ satisfies $\varphi^\lor=\varphi$. We denote by $\mathcal{F}_{\PP^n}^{o}(r,c)$ the moduli space of framed orthogonal instanton bundles and by $\mathcal{M}_{\PP^n}^{o}(r,c)$ the moduli space of orthogonal instanton bundles of trivial splitting type.
\end{enumerate}
\end{defi}

Our next result characterises symplectic and orthogonal instantons in terms of the data $(A,B,I,J,G,H)$.

\begin{lemma}\label{lem-dual}
Let the framed autodual instanton bundle $\EE$ be given by data $(A,B,I,J,G,H)$.
 \begin{enumerate}
  \item $\EE$ is symplectic if and only if $H$ is antisymmetric and $G$ is symmetric.
  \item $\EE$ is orthogonal if and only if $H$ is symmetric and $G$ is antisymmetric. 
 \end{enumerate}
\end{lemma}

\begin{proof}
By dualising the map $\varphi\colon\EE\rightarrow\EE^\lor$ one gets the following picture
\[
\begin{xy}
 \xymatrix{
 V\otimes \OO_{\PP}(-1) \ar[r]^{\alpha}\ar[d]^{G^\lor} & \widetilde{W}\otimes\OO_{\PP} \ar[r]^{\beta}\ar[d]^{F^\lor} & V\otimes \OO_{\PP}(1) \ar[d]^{-G^\lor} \\
 V^\lor\otimes \OO_{\PP}(-1) \ar[r]^{\beta^\lor} & \widetilde{W}^\lor\otimes\OO_{\PP} \ar[r]^{\alpha^\lor} & V^\lor\otimes \OO_{\PP}(1),
 }
\end{xy}
\]
i.e. the bundle map $\varphi^\lor\colon\EE\rightarrow\EE^\lor$ is given by the isomorphisms $-G^\lor\colon V\rightarrow V^\lor$ and $H^\lor\colon W\rightarrow W^\lor$.
Now $\varphi^\lor=-\varphi$ is equivalent to $G^\lor=G$ and $H^\lor=-H$. This gives the statement in the symplectic case. The orthogonal case is completely analogous. 
\end{proof}

The case of rank $2$ instanton bundles will play an important role in this paper, so let us take a closer look in this case. 

If $\rk(\EE)=2$, then $\bigwedge^2 \EE\simeq \OO_{\PP^n}$ and there is a natural isomorphism $\EE\simeq\EE^\lor$. It follows that the bundle map $\varphi\colon\mathcal{E}\rightarrow\mathcal{E}^\lor$ can be regarded as a section in $\Hom(\EE,\EE) \simeq H^0 (\EE^\lor \otimes \EE)$; then a symplectic structure on $\EE$ becomes a section in $H^0(\bigwedge^2 \EE)$, while an orthogonal structure on $\EE$ becomes a section in $H^0(S^2(\EE))$.

Note also that every rank $2$ instanton bundle is autodual and admits a symplectic structure, which is unique up to scaling. On the other hand, orthogonal structures on rank $2$ instanton bundles may not exist. Indeed, in this situation we get
$$
\Hom(\EE,\EE) \simeq H^0(\bigwedge^2 \EE) \oplus H^0(S^2(\EE)),
$$
therefore a rank $2$ instanton bundle is simple (hence stable) if and only if it admits no orthogonal structure. 

Since every instanton bundle of rank $2$ on $\PP^3$ is stable, one concludes that there are no orthogonal instanton bundles of rank $2$ on $\PP^3$, a particular case of the Farnik--Frapporti--Marchesi non-existence result mentioned at the Introduction.

If $\EE$ is not simple, it may admit, depending on $h^0(S^2(\EE))$, distinct (up to scaling) orthogonal structures. Furthermore, $\EE$ may also admit autodual structures which are neither symplectic, nor orthogonal: just consider pairs $(\EE,\varphi)$ such that $\varphi = (\varphi_1,\varphi_2) \in H^0(\bigwedge^2 \EE) \oplus H^0(S^2(\EE))$, with $\varphi_1 \neq 0$ and $\varphi_2 \neq 0$.


\section{Symplectic Bundles}\label{sym}

In this section, we will describe the moduli space of framed symplectic instanton bundles in terms of \emph{ADHM} data.

Being a special case of autodual instanton bundles, we know from Proposition \ref{prop-modaut}, that a framed symplectic instanton bundle gives rise to an extended \emph{ADHM} datum $(A,B,I,G,H)$, where $H$ satisfies diagram (\ref{H-Phi}). 

From Lemma \ref{lem-dual}, we obtain that $H$ is an antisymmetric isomorphism and $G$ is a symmetric one. Moreover, we consider the defined group action
\[ (g.G)^\lor=((g^\lor)^{-1}Gg^{-1})^\lor=(g^\lor)^{-1}Gg^{-1}=g.G \]
for any $g\in GL(V)$ and we see that the symmetry of $G$ is not affected by it. 

            
Let us take a closer look at the duality relations.
Since we are in the symplectic case, by Lemma (\ref{lem-dual}) $G$ is a symmetric isomorphism whereas $H$ is antisymmetric, and thus $GIH^{-1}+G^\lor I(H^\lor)^{-1}=GIH^{-1}-GIH^{-1}=0$ holds for all given $I$.
            
\begin{prop}
 For fixed $r>0$ and $c>0$ we have a bijection
 \[ \mathcal{F}_{\PP^n}^{s}(r,c) \cong \left\{ (A,B,I,G,H) \right\}/_{GL(V)} \]
 where $(A,B,I,G,H)$ satisfies the following:
 \begin{itemize}
  \item $(A,B,I,-H^{-1}I^\lor G)$ is a regular ADHM datum.
  \item $GA=A^\lor G$, $GB=B^\lor G$.
  \item $G^\lor =G$, $H^\lor=-H$.
 \end{itemize}
\end{prop}

Let us take a closer look at the framing $\Phi\colon\EE_{|l}\rightarrow W\otimes\OO_{l}$.
For this, we consider an action on $H$, which by construction is related to the framing of $\EE$. In order to do so, we consider elements $h$ of the symplectic group $Sp(W)$ and we define an action
\[ h.(I,H) := (Ih^{-1},(h^\lor)^{-1}Hh^{-1}). \]
It is known that for any symplectic pair $(W,H)$, given by a vector space and a symplectic bilinear form, we can always get, through the action defined before, that 
\[ H = \left( \begin{matrix}
                 0 & \mathbf{1} \\
                 -\mathbf{1} & 0 
               \end{matrix}\right) =:\Omega \]
which is known as the standard symplectic structure. Hence, if we consider bundles which are of \emph{trivial splitting type} (i.e. no framing is chosen) we can always assume to have $H$ in standard form. Recall that $\mathcal{M}_{\PP^n}^{s}(r,c)$ denotes the moduli space of symplectic instanton bundles of trivial splitting type, we hence obtain the following description.

\begin{prop}
For fixed $r>0$ and $c>0$ we have a bijection
 \[ 
 \mathcal{M}_{\PP^n}^{s}(r,c) \cong \left\{ (A,B,I,G) \right\}/_{GL(V)\times Sp(W)}
\]
where $(A,B,I,G)$ satisfies the following:
\begin{itemize}
  \item $(A,B,I,-\Omega^{-1}I^\lor G)$ is a regular ADHM datum.
  \item $GA=A^\lor G$, $GB=B^\lor G$.
  \item $G^\lor =G$.
\end{itemize}
\end{prop}

\begin{rem}
Certain moduli spaces of symplectic instantons have been much studied. The simplest case $\mathcal{M}_{\PP^2}^{s}(r,c)$, proved in \cite[Theorem 7.7]{ott} to be a nonsingular, irreducible variety of dimension $(r+2)c-{{r+1}\choose{2}}$, whenever non-empty. Another classical situation is
$\mathcal{M}_{\PP^3}^{s}(2,c)$, which, after recent results by Jardim--Verbitsky and Tikhomirov \cite{JV,Tik1,Tik2} is know to be a nonsingular, irreducible variety of dimension $8c-3$. The case $\mathcal{M}_{\PP^3}^{s}(r,c)$ for even $r$ is studied in \cite{BMT}. In higher dimension, 
$\mathcal{M}_{\PP^{2k+1}}^{s}(2k,c)$ is addressed in \cite{CO2,OS}.
\end{rem}


\section{Orthogonal Bundles}\label{ort}

In this Section, we will describe the moduli space in the orthogonal case, concluding with examples of rank $2$ orthogonal instantons bundles on $\PP^2$ and rank $4$ orthogonal instantons bundles on $\PP^3$.

Orthogonal instanton bundles have always atracted special attention, as it was very hard to find any examples. That was because in \cite{orth} the authors prove that orthogonal instanton bundles of rank $2n$ on $\PP^{2n+1}$, which represent the most studied case of instanton bundles, do not exist. We will extend this result in various directions. 

Similarly to the symplectic case, an orthogonal instanton bundle gives an extended \emph{ADHM} datum $(A,B,I,G,H)$ with $H\colon W\rightarrow W^\lor$ symmetric and $G\colon V\rightarrow V^\lor$ antisymmetric. Again, the antisymmetry of $G$ is not affected by the action of $Gl(V)$. The duality relation is again fulfilled for all $I$, since
\[ GIH^{-1}+G^\lor I(H^\lor)^{-1}= GIH^{-1}-GIH^{-1}=0. \] 


\begin{prop}
 For fixed $r>0$ and $c>0$ we have a bijection
 \[ \mathcal{F}_{\PP^n}^{o}(r,c) \cong \left\{ (A,B,I,G,H) \right\}/_{GL(V)} \]
 where $(A,B,I,G,H)$ satisfies the following:
 \begin{itemize}
  \item $(A,B,I,-H^{-1}I^\lor G)$ is a regular ADHM datum.
  \item $GA=A^\lor G$, $GB=B^\lor G$.
  \item $G^\lor=-G$, $H^\lor=H$.
 \end{itemize}
\end{prop}

Distinctly from the symplectic case, one cannot bring an orthogonal structure $H$ vector space $(W,H)$ into a standard form. Let the orthogonal group $O(W)$ act on the pair $(I,H)$ as 
\[ h.(I,H) := (Ih^{-1},(h^\lor)^{-1}Hh^{-1}). \]
Forgetting the framing of a framed orthogonal bundle amounts to quotienting $\mathcal{F}_{\PP^n}^{o}(r,c)$ by this action of $O(W)$, thus we obtain the following result. Again, recall that $\mathcal{M}_{\PP^n}^{o}(r,c)$ denotes the moduli space of orthogonal instanton bundles of trivial splitting type.

\begin{prop}
For fixed $r>0$ and $c>0$ we have a bijection
\[  \mathcal{M}_{\PP^n}^{o}(r,c) \cong \left\{ (A,B,I,G,H) \right\}/_{GL(V)\times O(W)} \]
where $(A,B,I,G,H)$ satisfies the following:
\begin{itemize}
  \item $(A,B,I,-H^{-1}I^\lor G)$ is a regular ADHM datum.
  \item $GA=A^\lor G$, $GB=B^\lor G$.
  \item $G^\lor=-G$, $H^\lor=H$.
\end{itemize}
\end{prop}

\begin{rem}
Recently, Abuaf and Boralevi proved that $\mathcal{M}_{\PP^2}^{o}(r,c)$ is an irreducible, nonsingular variety of dimension $(r-2)c-{{r}\choose{2}}$ for $r=c$ and $c\ge4$, and for $r=c-1$ and $c\ge8$, whenever non-empty, cf. \cite[Theorem 3.4]{AB}. To our knowledge, $\mathcal{M}_{\PP^n}^{o}(r,c)$ has not been previously studied for $n\ge3$. 
\end{rem}

While symplectic instanton bundles have been studied for many authors, the existence of orthogonal instanton bundles is less explored. Therefore, let us start by providing an easy, explicit example of an orthogonal instantons bundles of rank $2n$ and charge $1$ on $\PP^n$. Indeed, consider the monad
\[ \begin{xy}
 \xymatrix{
  \OO_{\PP^n}(-1) \ar[r]^{\alpha^t} & \OO^{2n+2}_{\PP^n} \ar[r]^{\alpha} & \OO_{\PP^n}(1)
  }
\end{xy} \]
where $[x_0:\ldots:x_n]$ are homogeneous coordinates on $\PP^n$ and 
\[ \alpha=(x_0,\ldots,x_n,i\cdot x_0,\ldots,i\cdot x_n). \]
Clearly $\alpha\alpha^t=0$ and $\alpha^t$ is injective in every point so that its cohomology bundle, denoted by $\EE$, is an orthogonal instanton bundle of $\rk(\EE)=2n$ and charge $1$.

Note that, by \cite[Corollary 4.1]{JdS}, $\EE$ is decomposable. Indeed, in light of \cite[Proposition 4.3]{JdS}, one can check that the monad above splits as a sum of two monads of the form
\[ \begin{xy}
 \xymatrix{
  \OO_{\PP^n}(-1) \ar[r]^{\tau^t} & \OO^{n+1}_{\PP^n} \ar[r]^{0} & 0
  }
\end{xy} \] 
and
\[ \begin{xy}
 \xymatrix{
  0 \ar[r]^{0} & \OO^{n+1}_{\PP^n} \ar[r]^{i\cdot\tau} & \OO_{\PP^n}(1)
  }
\end{xy} \]
where $\tau=\alpha=(x_0,\ldots,x_n)$. It follows that $\EE\simeq T_{\PP^n}(-1)\oplus\Omega_{\PP^n}(1)$, where $ T_{\PP^n}$ and $\Omega_{\PP^n}(1)$ denote the tangent and cotangent bundles of $\PP^n$, respectively.

It also follows that $\EE$ is not of trivial splitting type, thus it cannot be obtained via the \emph{ADHM} construction. Alternatively, note that if $l=\left\langle xy\right\rangle$ is the line through $x,y\in\PP^{n}$, then
 \[ \det(\alpha_l(x)\alpha_l(y)^t) = \sum_{i=0}^n x_iy_i+i^2x_iy_i = 0 , \]
thus the restriction $\EE|_{l}$ is not trivial (cf. \cite[Lemma II.4.2.3 and Remark II.4.2.4]{oss}.

In fact, one can show the following general non-existence result.

\begin{lemma}\label{no-c-odd}
If $\EE$ is an orthogonal instanton bundle of trivial splitting type, then $c=c_2(\EE)$ is even. In particular, $\mathcal{M}_{\PP^n}^{o}(r,c)$ is empty for $c$ odd.
\end{lemma}
\begin{proof}
 The \emph{ADHM} datum of $\EE$ gives an antisymmetric isomorphism $G\colon V\rightarrow V^\lor$, forcing $\dim V=c$ to be even.
\end{proof}

\begin{rem}
An alternative proof for the case $n=2$ is given in \cite[Proposition 3.1]{AB}.
\end{rem}


\subsection{Framed orthogonal instanton bundles on $\PP^2$}

In light of Lemma \ref{no-c-odd}, the simplest possible situation in which framed orthogonal instanton bundles may exist is for rank $2$ and charge $2$ on $\PP^2$. Our next result says that these do not exist either.

\begin{thm}
Framed orthogonal instanton bundles of rank $2$ and charge $2$ on $\PP^2$ do not exist. In other words, $\mathcal{M}_{\PP^2}^{o}(2,2)=\emptyset$
\end{thm}
\begin{proof}
In this case, all maps in the extended \emph{ADHM} datum $(A,B,I,J,G,H)$ are given by $(2\times 2)$-matrices. In order to construct an orthogonal instanton, we need $G^\lor=-G$. This means that the pair $(V,G)$ gives a symplectic vector space and hence we can choose a basis such that
 \[ G=\begin{pmatrix}
       0 & 1 \\
       -1 & 0
      \end{pmatrix}.
 \] 
 Setting 
 $A=\begin{pmatrix}
             a_1 & a_2 \\
             a_3 & a_4
            \end{pmatrix}, $
we compute the duality relation $GA=A^\lor G$ and get
\begin{align*}
 \begin{pmatrix}
       a_3 & a_4 \\
       -a_1 & -a_2
      \end{pmatrix}
 =GA=A^\lor G
 = \begin{pmatrix}
      a_1 & a_3 \\
      a_2 & a_4
     \end{pmatrix}
 \begin{pmatrix}
      0 & 1 \\
      -1 & 0
     \end{pmatrix}
 = \begin{pmatrix}
      -a_3 & a_1 \\
      -a_4 & a_2
     \end{pmatrix}, 
\end{align*}
     hence 
     $
     A=\begin{pmatrix}
          a & 0 \\
          0 & a
         \end{pmatrix}$
 and similarly $B=\begin{pmatrix}
                    b & 0 \\
                    0 & b
                   \end{pmatrix}.
$
Now $A$ and $B$ commute, so the \emph{ADHM} equation reduces to $IJ=0$. The maps $I$ and $J$ have the same rank since $GI=J^\lor H$, and combining this with the \emph{ADHM} equation we get that $\rk(I),\rk(J)\leq 1$.
 
 Choosing homogeneous coordinates $[z:x:y]$ on $\PP^2$ the monad maps take the form
 \[ \alpha=\begin{pmatrix}
            az+x & 0 \\
            0 & az+x \\
            bz+y & 0 \\
            0 & bz+y \\
            i_1z & i_2z \\
            i_3z & i_4z
           \end{pmatrix}, \hspace{.5cm}
     \beta=\begin{pmatrix}
            -bz-y & 0 & az+x & 0 & j_1z & j_2z \\
            0 & -bz-y & 0 & az+x & j_3z & j_4z
           \end{pmatrix},
 \]
 where $I=\begin{pmatrix}
           i_1 & i_2 \\
           i_3 & i_4
          \end{pmatrix}
$, $J=\begin{pmatrix}
       j_1 & j_2 \\
       j_3 & j_4
      \end{pmatrix}.
$ We let $p=[-1:a:b]\in\PP^2$ and get for the fibre maps 
\[ \alpha_p = \begin{pmatrix}
               0 \\
               0 \\
               -I
              \end{pmatrix},\hspace{.5cm}
    \beta_p = \begin{pmatrix}
               0 & 0 & J
              \end{pmatrix}.
 \]
 Since $\rk(I),\rk(J)\leq 1$, we see that $\alpha_p$ can never be injective and $\beta_p$ can never be surjective, which concludes the proof.
\end{proof}

\bigskip

We will now prove that it is possible to find rank $2$ orthogonal instanton bundles of charge $4$ on $\PP^2$, i.e. $\mathcal{M}_{\PP^2}^{o}(2,4) \neq\emptyset$, using the conditions used to describe their moduli space.

We know by now that this means considering the commutative diagram
$$
\begin{xy}
 \xymatrix{
 V\otimes \OO_{\PP}(-1) \ar[r]^{\alpha}\ar[d]^{-G} & \widetilde{W}\otimes\OO_{\PP} \ar[r]^{\beta}\ar[d]^{F} & V\otimes \OO_{\PP}(1) \ar[d]^{G} \\
 V^\lor\otimes \OO_{\PP}(-1) \ar[r]^{\beta^\lor} & \widetilde{W}^\lor\otimes\OO_{\PP} \ar[r]^{\alpha^\lor} & V^\lor\otimes \OO_{\PP}(1)
 }
\end{xy}
$$
with the following relations
\begin{equation*}
\begin{array}{c}
 GB = B^\lor G  \\
 GA = A^\lor G \\
 HJ = -I^\lor G \\
 GI = J^\lor H
 \end{array}
\end{equation*}
and moreover $H$ is symmetric and $G$ antisymmetric. Being $(V,G)$ a symplectic vector space we can fix a basis of $V$ such that
$$
G =
\left(
\begin{array}{cccc}
0 & 0 & 1 & 0 \\
0 & 0 & 0 & 1 \\
-1 & 0 & 0 & 0 \\
0 & -1 & 0 & 0
\end{array}
\right)
$$
Using the first two relations in (\ref{duality}), we obtain that $A$ and $B$ have the following form
$$
A =
\left(
\begin{array}{cccc}
a_1 & a_2 & 0 & a_4 \\
a_5 & a_6 & -a_4 & 0 \\
0 & a_{10} & a_1 & a_5 \\
-a_{10} & 0 & a_2 & a_6
\end{array}
\right)
$$
and
$$
B =
\left(
\begin{array}{cccc}
b_1 & b_2 & 0 & b_4 \\
b_5 & b_6 & -b_4 & 0 \\
0 & b_{10} & b_1 & b_5 \\
-b_{10} & 0 & b_2 & b_6
\end{array}
\right)
$$
In order to simplify computation, we fix $H$ to be the identity matrix and 
$$
J =
\left(
\begin{array}{cccc}
j_1 & 0 & 0 & 0 \\
0 & 0 & 0 & j_8 
\end{array}
\right)
$$
and recover $I$ from the remaining relations.

Let us compute the \emph{ADHM} equation (which must have all zero entries) with the obtained matrices, and we have a further matrix
$$
AB - BA + IJ = \left(
\begin{array}{cc}
 -a_5b_2+a_{10}b_4+a_2b_5-a_4b_{10}  & -a_2b_1+a_1b_2-a_6b_2+a_2b_6     \\
a_5b_1-a_1b_5+a_6b_5-a_5b_6    &  a_5b_2+a_{10}b_4-a_2b_5-a_4b_{10}   \\
2a_{10}b_5-2a_5b_{10}+j_1^2      &     -a_{10}b_1+a_{10}b_6+a_1b_{10}-a_6b_{10} \\
      -a_{10}b_1+a_{10}b_6+a_1b_{10}-a_6b_{10} & -2a_{10}b_2+2a_2b_{10}     \\
\end{array}
\right.
$$
$$
\left.
\begin{array}{cc}
2a_4b_2-2a_2b_4  &              -a_4b_1+a_1b_4-a_6b_4+a_4b_6 \\
-a_4b_1+a_1b_4-a_6b_4+a_4b_6  & 2a_5b_4-2a_4b_5-j_8^2 \\
  a_5b_2-a_{10}b_4-a_2b_5+a_4b_{10}  & -a_5b_1+a_1b_5-a_6b_5+a_5b_6\\
a_2b_1-a_1b_2+a_6b_2-a_2b_6 & -a_5b_2-a_{10}b_4+a_2b_5+a_4b_{10}   
      \end{array}
\right)
$$
In order to obtain the orthogonal bundle, we must find solutions in the $a$'s and $b$'s such that $j_1$ and $j_8$ are not zero and, to prove the existence of such a solution, we define the ideal $Y$ generated by all entries of the matrix without a summand in either $j_1$ or $j_8$. After that, we consider the ideal $X$ defined as
$$
X = \langle Y , a_{10} b_5-a_5 b_{10}, a_5 b_4-a_4 b_5 \rangle.
$$
It is possible to prove that $X$ and $Y$ are not the same ideal, which means that we can find solutions of the \emph{ADHM} equation with $j_1$ and $j_8$ different from zero.

In order to find the orthogonal instanton bundle we have to look for regular solutions of the \emph{ADHM} equation. Recall that this is equivalent to show that we have a solution such that there are no proper subspaces of the image of $I$ or the kernel of $J$ (that are both subspaces of $V$) which are not invariant through the linear applications $A$ and $B$.

To this end, note that 
\[
 I=G^{-1}J^\lor H=\begin{pmatrix}
               0 & 0 \\
               0 & -j_8 \\
               j_1 & 0 \\
               0 & 0
              \end{pmatrix},
\]
and hence $\im I = \langle (0 1 0 0)^t, (0 0 1 0)^t \rangle=\ker J$. For the regularity of $(A,B,I,J)$ we need that, for all subspaces $S\subset V=\mathbb{C}^4$ with $\im I\subset S \subset \ker J$, that $S$ is not $A,B$-invariant. And since
\[ 
 A\cdot\begin{pmatrix}
        0 \\
        x \\
        y \\
        0
       \end{pmatrix}=
       \begin{pmatrix}
        a_2x \\
        a_6x-a_4y \\
        a_{10}x+a_1y \\
        a_2y
       \end{pmatrix},
\]
it is sufficient to ask that both $a_2$ and $b_2$ are not zero. As before, consider the ideal 
$$
W = \langle Y, a_2, b_2 \rangle
$$
which is not equal to $Y$; therefore we manage to find the required solution.

All the steps described during this part can be proven using Macaulay2; the code we used is included in Appendix \ref{ex-p2}.

Finally, using the same algorithm as in the charge $4$ case, one can also show the existence of regular solutions in case $c=6$, although the computations become very large. We are confident that it also works for higher charges.

\subsection{Framed orthogonal instanton bundles on $\PP^3$}

As noted before, there are no orthogonal bundles of rank $2$ on $\PP^3$. The next possibility would occur in rank $4$ and charge $2$. Indeed, it is possible to adapt the previous algorithm in order to obtain examples of framed orthogonal instanton bundles in higher dimensional projective spaces; in this Section we will show that indecomposable orthogonal instanton bundles of rank $4$ and charge $2$ on $\PP^3$ do exist.

Our goal is to find regular solutions of the \emph{ADHM} equations which satisfy the orthogonal relations that we have already described. Notice that in this case we will have to consider a 1-dimensional \emph{ADHM} datum, i.e. we are looking for matrices $A_k, B_k \in \End(V)$, $I_k \in \Hom(W,V)$ and $J_k \in \Hom(V,W)$, for $k=1,2$.

As before, we observe that $(V,G)$ is a symplectic vector space, therefore we can fix

$$
G= 
\left(
\begin{array}{cc}
0 & 1 \\
-1 & 0
\end{array}
\right)
$$

In order to satisfy the orthogonal relations, we consider

$$
A_0 =  \left(
\begin{array}{cc}
a_1 & a_2 \\
a_3 & -a_1
\end{array}
\right)  \:\:\:\:\: A_1 = \left(
\begin{array}{cc}
a_5 & a_6 \\
a_7 & -a_5
\end{array}
\right)
$$

and

$$
B_0 =  \left(
\begin{array}{cc}
b_1 & b_2 \\
b_3 & -b_1
\end{array}
\right)  \:\:\:\:\: B_1 = \left(
\begin{array}{cc}
b_5 & b_6 \\
b_7 & -b_5
\end{array}
\right)
$$

and, to simplify computations as usual, we take

$$
J_0 =  \left(
\begin{array}{cc}
j_1 & 0 \\
0 & 0 \\
0 & 0 \\
0 & j_8
\end{array}
\right)  \:\:\:\:\: J_1 = \left(
\begin{array}{cc}
0 & 0 \\
j_9 & 0 \\
0 & j_{16} \\
0 & 0
\end{array}
\right)
$$

We ask for the equations

$$
A_0 B_0 - B_0 A_0 + I_0 J_0 = \left(
\begin{array}{cc}
b_3a_2-b_2a_3 & 2b_2a_1-2b_1a_2-j_8^2 \\
-2b_3a_1+2b_1a_3+j_1^2 & -b_3a_2+b_2a_3
\end{array}
\right)
$$

$$
A_1 B_1 - B_1 A_1 + I_1 J_1 = \left(
\begin{array}{cc}
 b_7a_6-b_6a_7 & 2b_6a_5-2b_5a_6-j_{16}^2 \\
-2b_7a_5+2b_5a_7+j_9^2 & -b_7a_6+b_6a_7
\end{array}
\right).
$$

$$\begin{array}{c}
A_0 B_1 - B_1 A_0 + B_0 A_1 - A_1 B_0 + I_0 J_1 + I_1 J_0 = \vspace{2mm}\\
=  \left(
\begin{array}{cc}
b_7a_2-b_6a_3-b_3a_6+b_2a_7 & 2b_6a_1-2b_5a_2-2b_2a_5+2b_1a_6 \\
-2b_7a_1+2b_5a_3+2b_3a_5-2b_1a_7 & -b_7a_2+b_6a_3+b_3a_6-b_2a_7
\end{array}
\right)
\end{array} \vspace{2mm}
$$

to vanish, which implies the vanishing of the \emph{ADHM} equation. Moreover we want the coefficients $j_1, j_8, j_9$ and $j_{16}$ not to be zero and the coefficients $a_3,a_7,b_3,b_7$ to be zero in order to get a regular solution. Hence we want the following relations
$$
(b_2 a_1-b_1 a_2, -b_3 a_1+b_1 a_3, b_6 a_5-b_5 a_6, -b_7a_5+b_5 a_7, a_2)
$$

not to belong to the ideal
$$
\begin{array}{rl}
X:= & \langle b_3 a_2-b_2 a_3,b_7 a_6-b_6 a_7, -b_7 a_1 + b_5 a_3 + b_3 a_5 - b_1 a_7, b_6 a_1-b_5 a_2 - b_2 a_5 + b_1a_6,  \\ &  b_7 a_2-b_6 a_3-b_3 a_6+b_2 a_7,a_3,a_7,b_3,b_7 \rangle.
\end{array}
$$

Finally, we also check that one can get solutions in which the coefficients $a_2,a_6,b_2,b_6$ are not zero. This means that there exists a solution with the matrices $A_0, A_1, B_0, B_1$ not of a diagonal form,
therefore the corresponding linear monad is indecomposable, and thus its cohomology bundle is the desired indecomposable orthogonal instanton bundle.

All the steps described during this part can be proven using Macaulay2; the code we used is included in Appendix \ref{ex-p3}.


\appendix 

\section{Macaulay2 code for orthogonal bundles on $\PP^2$}\label{ex-p2}

Below is the Macaulay2 code we used to establish the existence of rank $2$ orthogonal instanton bundles of charge $4$ on $\PP^2$.

\begin{verbatim}
R=QQ[a_0..a_59,b_0..b_59, h_1, h_2, h_3,j_1..j_8,i_1..i_8]
A=matrix{{a_1,a_2,0,a_4},{a_5,a_1,-a_4,0},{0,a_10,a_1,a_5},{-a_10,0,a_2,a_1}}
B=matrix{{b_1,b_2,0,b_4},{b_5,b_1,-b_4,0},{0,b_10,b_1,b_5},{-b_10,0,b_2,b_1}}
G=matrix{{0,0,1,0},{0,0,0,1},{-1,0,0,0},{0,-1,0,0}}
iG=inverse G
J=matrix{{j_1,0,0,0},{0,0,0,j_8}}
TJ=transpose J
H=matrix{{1,0},{0,1}}
I=iG*TJ*H
M=A*B-B*A + I*J
Q=ideal{M}
mingens Q
Y=ideal(a_4*b_2-a_2*b_4, a_10*b_2-a_2*b_10, a_5*b_2-a_2*b_5, 
-a_4*b_10+a_10*b_4,  a_10*b_1-a_10*b_6-a_1*b_10+a_6*b_10,
a_5*b_1-a_1*b_5+a_6*b_5-a_5*b_6, a_4*b_1-a_1*b_4+a_6*b_4-a_4*b_6,
a_2*b_1-a_1*b_2+a_6*b_2-a_2*b_6 )
mingens Y
X=ideal(Y, a_10*b_5-a_5*b_10, a_5*b_4-a_4*b_5 )
mingens X
X==Y
W=ideal(Y, a_2, b_2 )
W==Y
\end{verbatim}


\section{Macaulay2 code for orthogonal bundles on $\PP^3$} \label{ex-p3}

Below is the Macaulay2 code we used to establish the existence of indecomposable rank $4$ orthogonal instanton bundles of charge $2$ on $\PP^3$.

\begin{verbatim}
R:=QQ[x_0,x_1,x_2,x_3,b_1..b_8,a_1..a_8,j_1..j_16,i_1..i_16,h_1..h_16]
A0:=matrix{{a_1,a_2},{a_3,-a_1}}
A1:=matrix{{a_5,a_6},{a_7,-a_5}}
B0:=matrix{{b_1,b_2},{b_3,-b_1}}
B1:=matrix{{b_5,b_6},{b_7,-b_5}}
G:=matrix{{0,1},{-1,0}}
H:=matrix{{h_1,h_2,h_3,h_4},{h_5,h_6,h_7,h_8},{h_9,h_10,h_11,h_12},
{h_13,h_14,h_15,h_16}}
HI:=id_(R^4)
J0:=matrix{{j_1,0},{0,0},{0,0},{0,j_8}}
J1:=matrix{{0,0},{j_9,0},{0,j_16},{0,0}}
A:=A0*x_0 + A1*x_1
B:=B0*x_0 + B1*x_1
J:=J0*x_0 + J1*x_1
I0 := transpose G * transpose J0 * HI
I1 := transpose G * transpose J1 * HI
I := transpose G * transpose J * HI
G*A - transpose A * transpose G
G*B - transpose B * transpose G
A*B - B*A + I*J
A0*B0 - B0*A0 + I0*J0
A1*B1 - B1*A1 + I1*J1
A0*B1 - B1*A0 + B0*A1 - A1*B0 + I0*J1 + I1*J0
X:=ideal{b_3*a_2-b_2*a_3,b_7*a_6-b_6*a_7, -b_7*a_1 + b_5*a_3 + b_3*a_5 - b_1*a_7, 
b_6*a_1-b_5*a_2 - b_2*a_5 + b_1*a_6, b_7*a_2-b_6*a_3-b_3*a_6+b_2*a_7,
a_3,a_7,b_3,b_7}
Y:=ideal{X, b_2*a_1-b_1*a_2, -b_3*a_1+b_1*a_3, b_6*a_5-b_5*a_6, -b_7*a_5+b_5*a_7}
X==Y
Z:=ideal{Y,a_2,a_6,b_2,b_6}
Z==X
Z==Y
\end{verbatim}

\bibliography{instantons2}{}
\bibliographystyle{siam}

\end{document}